\documentclass[12pt, reqno]{amsart}
\usepackage{amsmath, amsthm, amscd, amsfonts, amssymb, graphicx, color}
\usepackage[bookmarksnumbered, colorlinks, plainpages]{hyperref}
\hypersetup{colorlinks=true,linkcolor=red, anchorcolor=green, citecolor=cyan, urlcolor=red, filecolor=magenta, pdftoolbar=true}

\newtheorem{theorem}{Theorem}[section]
\newtheorem{lemma}[theorem]{Lemma}
\newtheorem{proposition}[theorem]{Proposition}
\newtheorem{corollary}[theorem]{Corollary}
\theoremstyle{definition}
\newtheorem{definition}[theorem]{Definition}
\newtheorem{example}[theorem]{Example}

\theoremstyle{remark}

\numberwithin{equation}{section}

\begin{document}
	
	\setcounter{page}{1}
	
	\title{Exploring $K$-Riesz and $K$-$g$-Riesz Bases in the Context of Hilbert Pro-$C^*$-Modules }
	
	\author{Roumaissae Eljazzar$^{1*}$ \MakeLowercase{and} Mohamed Rossafi$^2$}
	
	\address{$^{1}$Department of Mathematics, Faculty Of Sciences, University of Ibn Tofail, Kenitra, Morocco}
	\email{ro.eljazzar@gmail.com ;roumaissae.eljazzar@uit.ac.ma}
	
	\address{$^{2}$Department of Mathematics, Faculty of Sciences Dhar El Mahraz, University Sidi Mohamed Ben Abdellah, Fes, Morocco}
	\email{mohamed.rossafi@usmba.ac.ma; mohamed.rossafi@uit.ac.ma}

	
	
	\date{$^{*}$Corresponding author}
	
	\setcounter{page}{1}
	
	\title[Exploring $K$-Riesz and $K$-$g$-Riesz Bases in the Context of Hilbert Pro-$C^*$-Modules ]{Exploring $K$-Riesz and $K$-$g$-Riesz Bases in the Context of Hilbert Pro-$C^*$-Modules}
	
	\maketitle
	
\begin{abstract}
	In the present research, we embark on a comprehensive inquiry into \(K\)-Riesz bases and \(K\)-g Riesz bases as they manifest within pro-\(C^*\)-Hilbert modules. Adopting a unique approach, we interpret the structure of \(K\)-Riesz bases through the lens of a bounded below operator, complemented by the prototypical orthonormal basis of the Hilbert module \(\mathcal{X}\). In parallel, our investigation elucidates nuanced traits and pioneering perspectives of \(K\)-g-Riesz bases, steered by a bounded surjective operator and a corresponding \(g\)-orthonormal basis for \(\mathcal{X}\). As a culmination of our efforts, we present an enriched understanding of the dynamic interactions between \(K\)-Riesz bases and \(K\)-g-Riesz bases in this mathematical context.
\end{abstract}

\section{Introduction}
%In 1952, Duffin and Schaefer pioneered the concept of frames for Hilbert spaces as a tool for the analysis of non-harmonic Fourier series [1]. Frame theory has since evolved into a vibrant, rapidly evolving field with broad applications spanning across mathematics and engineering.This topic is pertinent across various fields such as operator theory, harmonic analysis, wavelet theory, wireless communication, data transmission with erasures, signal processing, and image processing, among others, underscoring its dynamic and captivating nature.

The realm of frame theory in Hilbert spaces traces its roots back to the groundbreaking work of Duffin and Schaefer in 1952, initiated as an endeavor to comprehend non-harmonic Fourier series \cite{Duf}. Over the years, this discipline has blossomed, gathering momentum and expanding its scope. The plethora of applications encompassing mathematical domains such as operator theory and harmonic analysis, along with engineering fields ranging from wireless communication to image processing, underlines the immense potential and versatility of frame theory.

Modern technological advances underscore the practical implications of deepening our understanding in this domain. As data processing becomes increasingly sophisticated, the tools and frameworks we employ must similarly evolve. The richness of frame theory, with its diverse applications, positions it as a valuable asset in an era of complex computation and analysis. Its connection to wireless communication and image processing, for instance, highlights how mathematical abstractions find tangible, real-world implementations that significantly impact our daily lives.

However, as with any matured area of research, the quest for deeper insights and generalizations is a natural progression. In recent mathematical literature, the study of g-frames and g-Riesz bases within Hilbert spaces has gained significant attention. Sun's seminal work \cite{Sun} not only introduced these concepts but also delved into an exploration of their intrinsic properties.

The pursuit of generalized constructs, such as the $K$-Riesz bases, speaks to the broader goals of mathematical research: to encapsulate complex phenomena within more comprehensive and versatile frameworks. By doing so, we not only gain a richer understanding of the subject at hand but also pave the way for interdisciplinary applications and insights.

Taking inspiration from the pioneering work of Sheraki and Abdollahpour \cite{Shekari}, who introduced the notions of $K$-Riesz bases and $K$-g-Riesz bases in standard Hilbert spaces, our paper ventures into new territory. Specifically, we present the idea of $K$-Riesz basis and $K$-g-Riesz basis in Hilbert pro-$C^*$-module, positioning them as a natural generalization of the former constructs.

Our paper is structured for clarity and progressive exploration. In Section 2, we provide a concise overview of the foundational definitions and properties related to Hilbert pro-$C^*$-modules, frames, $K$-frames, g-frames, $K$-g-frames, Riesz bases, and g-Riesz bases as they stand in Hilbert pro-$C^*$-modules. Moving forward to Section 3, our focus narrows down to the construction and characterization of the $K$-Riesz basis within the Hilbert pro-$C^*$-module setting. Finally, in Section 4, we delve deep into the study of $K$-g-Riesz bases, illuminating their unique properties and significance in the Hilbert pro-$C^*$-module context.

\section{Preliminaries}
\begin{definition}\cite{Philips}
	Given a locally $C^{\ast}$-algebra, denoted as $\mathcal{A}$, a pre-Hilbert module over this algebra is a complex vector space, termed $\mathcal{X}$. This space also acts as a left $\mathcal{A}$-module in harmony with the complex algebra structure. It's equipped with an inner product, $\langle ., .\rangle: \mathcal{X} \times \mathcal{X} \rightarrow \mathcal{A}$, which is linear with respect to both $\mathbb{C}$ and $\mathcal{A}$ in the first argument and satisfies:
	
	\begin{enumerate}
		\item 
		For any vectors $\xi, \eta \in \mathcal{X}$, it is true that $\langle \xi, \eta\rangle^{*} = \langle \eta, \xi\rangle$.
		
		\item 
		It's always the case that $\langle \xi, \xi\rangle \geq 0$ for every vector $\xi \in \mathcal{X}$.
		
		\item 
		A vector $\xi$ in $\mathcal{X}$ fulfills $\langle \xi, \xi\rangle = 0$ if and only if that vector $\xi = 0$.
	\end{enumerate}
	
	Furthermore, if $\mathcal{X}$ is complete in terms of the topology induced by the set of seminorms:
	$$
	\bar{p}_{\mathcal{X}}(\xi) = \sqrt{p_{\alpha}(\langle \xi, \xi\rangle)}, \quad \xi \in \mathcal{X}, \text{ and } p \in S(\mathcal{A}),
	$$
	then we refer to $\mathcal{X}$ as either a Hilbert $\mathcal{A}$-module or, alternatively, a Hilbert pro-$C^{\ast}$-module over $\mathcal{A}$.
\end{definition}

\begin{definition}
	A mapping between $\mathcal{A}$-modules, denoted as  $T: \mathcal{X} \to \mathcal{Y}$, is termed adjointable if an accompanying map $T^{\ast}: \mathcal{Y} \to \mathcal{X}$ exists, fulfilling the relation: $\langle T\xi, \eta \rangle = \langle \xi, T^{\ast}\eta \rangle$ for each $\xi \in \mathcal{X}$ and $\eta \in \mathcal{Y}$. Such a map is designated as bounded when, for every $p$ in $S(\mathcal{A})$, a positive constant $M_{p}$ can be found ensuring that $\bar{p}_{\mathcal{Y}}(T\xi) \leq M_{p} \bar{p}_{\mathcal{X}}(\xi)$ for all vectors $\xi$ in $\mathcal{X}$.
	\\
	The collection of all adjointable operators transitioning from $\mathcal{X}$ to $\mathcal{Y}$ is represented as $Hom_{\mathcal{A}}^{\ast}(\mathcal{X}, \mathcal{Y})$. When referring to operators that map from $\mathcal{X}$ to itself, we use the notation $Hom_{\mathcal{A}}^{\ast}(\mathcal{X})=Hom_{\mathcal{A}}^{\ast}(\mathcal{X}, \mathcal{X})$.
\end{definition}

In this study, we regard \(\mathcal{X}\) to be a Hilbert \(\mathcal{A}\)-module. Let's take a set \(\left\{\mathcal{X}_i\right\}_{i \in I}\) of Hilbert \(\mathcal{A}\)-modules, where \(I\) can be a finite or countably infinite index set. We introduce \(l^2\left(\left\{\mathcal{X}_i\right\}_{i \in I}\right)\) as its associated Hilbert \(\mathcal{A}\)-module.

\[
l^2\left(\left\{\mathcal{X}_i\right\}_{i \in I}\right) = \left\{ \left\{x_i\right\}_{i \in I}: x_i \in \mathcal{X}_i \, \forall i \in I, \text{ and } \sum_{i \in I} \sqrt{p_\alpha \left\langle x_i, x_i \right\rangle_{\mathcal{A}}} \text{ is convergent } \right\},
\]
where the \(\mathcal{A}\)-inner product is given by \(\left\langle \left\{x_i\right\}_{i \in I}, \left\{y_i\right\}_{i \in I} \right\rangle = \sum_{i \in I} \left\langle x_i, y_i \right\rangle_{\mathcal{A}}\), applicable for all \(\left\{x_i\right\}_{i \in I}, \left\{y_i\right\}_{i \in I} \in l^2\left(\left\{\mathcal{X}_i\right\}_{i \in I}\right)\).

Furthermore, consider \(l^2(\mathcal{A})\) as the Hilbert \(\mathcal{A}\)-module described by

\begin{equation}
l^2(\mathcal{A}) = \left\{ \left\{\alpha_i\right\}_{i \in I} \subset \mathcal{A}: \text{ the series } \sum_{i \in I} \sqrt{p(\alpha_i \alpha_i^*)} \text{ converges for each semi-norm } p \in S(\mathcal{A}) \right\},
\end{equation}

and its \(\mathcal{A}\)-inner product is \(\left\langle \left\{\alpha_i\right\}_{i \in I}, \left\{\beta_i\right\}_{i \in I} \right\rangle = \sum_{i \in I} \alpha_i \beta_i^*\), for all \(\left\{\alpha_i\right\}_{i \in I}, \left\{\beta_i\right\}_{i \in I} \in l^2(\mathcal{A})\).

\begin{proposition}\cite{Alizadeh}\label{prop2.2}
	Suppose \(\mathcal{A}\) is a pro-\(C^*\)-algebra, and \(\mathcal{X}\) and \(\mathcal{Y}\) are Hilbert \(\mathcal{A}\)-modules. Let \(T\) be a bounded operator in \(Hom_{\mathcal{A}}^*(\mathcal{X}, \mathcal{Y})\). Then the subsequent conditions are equivalently true:
	\begin{enumerate}
		\item[(i)]   \(T\) is onto;
		\item[(ii)]  The adjoint \(T^*\) has a lower bound in \(Hom_{\mathcal{A}}^*(\mathcal{X}, \mathcal{Y})\), meaning there exists a positive constant \(m\) such that for all \(p\) in \(S(\mathcal{A})\) and for every \(\xi \in \mathcal{Y}\), we have \(m \bar{p}_{\mathcal{Y}}(\xi) \leq \bar{p}_{\mathcal{X}}\left(T^* \xi\right)\);
		\item[(iii)] The adjoint \(T^*\) possesses a lower bound with respect to the inner product, in other words, there is a positive value \(m^{\prime}\) for which \(m^{\prime}\langle \xi, \xi\rangle \leq\left\langle T^* \xi, T^* \xi\right\rangle\) holds for every \(\xi \in \mathcal{Y}\).
	\end{enumerate}
\end{proposition}

\begin{proposition}\cite{Azhini}. \label{Prop2.4}
	Consider a Hilbert module \(\mathcal{X}\) over the pro-\(C^{*}\)-algebra \(\mathcal{A}\). Assume that \(T\) is a inversible element in \(Hom_{\mathcal{A}}^{\ast}(\mathcal{X})\) and both maintain boundedness uniformly. For any given \(\xi \in \mathcal{X}\), the following inequalities hold:
	\[
	\frac{\langle \xi, \xi\rangle}{\|T^{-1}\|_{\infty}^{2}} \leq \langle T \xi, T \xi \rangle \leq \|T\|_{\infty}^{2}\langle \xi, \xi\rangle
	\].
\end{proposition}

\begin{definition}\cite{Joita}
	Consider a sequence $\left\{\xi_{i}\right\}_{i}$ within $M(\mathcal{X})$. This sequence is deemed a canonical frame of multipliers for $\mathcal{X}$ if, for any $\xi \in \mathcal{X}$, the series $\sum_{i}\left\langle\xi, \xi_{i}\right\rangle_{M(\mathcal{X})}\left\langle \xi_{i}, \xi\right\rangle_{M(\mathcal{X})}$ attains convergence in $\mathcal{A}$. Moreover, there exist two positive constants, namely $C$ and $D$, such that
	\[
	C\langle\xi, \xi\rangle_{\mathcal{X}} \leq \sum_{i}\left\langle\xi, \xi_{i}\right\rangle_{M(\mathcal{X})}\left\langle \xi_{i}, \xi\right\rangle_{M(\mathcal{X})} \leq D\langle\xi, \xi\rangle_{\mathcal{X}}
	\]
	for every $\xi \in \mathcal{X}$. When both $D$ and $C$ equal 1, the sequence $\left\{\xi_{i} \right\}_{i}$ is recognized as a normalized canonical frame of multipliers.
	
	Specifically, if the following inequality holds,
	\[
	\sum_{i}\left\langle\xi, \xi_{i}\right\rangle_{M(\mathcal{X})}\left\langle \xi_{i}, \xi\right\rangle_{M(\mathcal{X})} \leq D\langle\xi, \xi\rangle_{\mathcal{X}} \; \; \forall \xi \in \mathcal{X}
	\]
	then we term the sequence $\left\{\xi_{i}\right\}_{i\in I}$ a Bessel sequence.
\end{definition}
Let's first establish a foundational understanding of the Riesz basis concept.

\begin{definition}
	Consider a Hilbert space \(\mathcal{X}\) within a pro-$C^*$-module. A frame \(\left\{\xi_i\right\}_{i \in I}\) is termed a Riesz basis for \(\mathcal{X}\) if it meets the following criteria:
	\begin{enumerate}
		\item \(\xi_i \neq 0\) for all \(i \in I\).
		\item For any \(\mathcal{A}\)-linear combination 
		\[\sum_{j \in S} \alpha_j \xi_j\]
		with the coefficients \(\left\{\alpha_j: j \in S\right\} \subseteq \mathcal{A}\) and a subset \(S \subseteq I\) that results in zero, each term \(\alpha_j \xi_j\) must also be zero.
	\end{enumerate}
\end{definition}

%\begin{definition}
	% A frame $\left{f_i\right}_{i \in I}$ for $\mathcal{H}$ is said to be a Riesz basis for $\mathcal{H}$ if it satisfies the following conditions:(i) $f_i \neq 0$ for all $i \in I$.(ii) If an $\mathcal{A}$-linear combination $\sum_{j \in S} a_j f_j$ with coefficients $\left{a_j: j \in S\right} \subseteq \mathcal{A}$ and $S \subseteq I$ equals zero, then every term $a_j f_j$ in the sum is equal to zero.In this definition, $\mathcal{H}$ is considered as a Hilbert pro-C*-module embedded within a Hilbert $C^*$-module, and the frame is defined accordingly.
%\end{definition}

%Definition 2.5. [3] A frame $\left{f_i\right}{i \in I}$ for $\mathcal{H}$ is said to be a Riesz basis for $\mathcal{H}$ when $\mathcal{H}$ is a Hilbert pro-C*-module contained within a Hilbert C*-module, if it satisfies the following conditions:
%(i) $f_i \neq 0$ for all $i \in I.$
%(ii) If an $\mathcal{A}$-linear combination $\sum{j \in S} a_j f_j$ with coefficients $\left{a_j: j \in S\right} \subseteq \mathcal{A}$ and $S \subseteq I$ is equal to zero, then every summand $a_j f_j$ is equal to zero.

\begin{theorem}
	Let $\left\{\xi_j\right\}_{j \in J}$ be a frame of a finitely or countably generated Hilbert $\mathcal{A}$-module $\mathcal{X}$ over a unital pro-$C^*$-algebra $\mathcal{A}$. Then $\left\{\xi_j\right\}_{j \in J}$ is a Riesz basis if and only if $\xi_n \neq 0$ and $Q_n\left(Ran\left(T\right)\right) \subseteq Ran\left(T\right)$ for all $n \in J$, where $T$ is the analysis operator of $\left\{\xi_j\right\}_{j \in J}$.
\end{theorem}

\begin{proof}
	Suppose first that $\left\{\xi_j\right\}_{j \in J}$ is a Riesz basis. Note that for any $\alpha=\left\{\alpha_j\right\}_{j \in J}$ in $l^2(\mathcal{A})$, if $\sum_{j \in J} \alpha_j \xi_j=0$, then $\alpha_j \xi_j=0$ for all $j \in J$. Now for any $\alpha=\left\{\alpha_j\right\}_{j \in J} \in Ran\left(T\right)^{\perp}$ and $\xi \in \mathcal{X}$, we have
	\[
	0=\sum_{j \in J}\left\langle \xi, \xi_j\right\rangle \alpha_j^*=\left\langle \xi, \sum_{j \in J} \alpha_j \xi_j\right\rangle .
	\]
	And so $\alpha_j \xi_j=0$ holds for all $j$. This implies that
	\[
	\sum_{j=1}^n\left\langle \xi, \xi_j\right\rangle \alpha_j^*=\sum_{j=1}^n\left\langle \xi, \alpha_j \xi_j\right\rangle=0 .
	\]
	It follows that $\alpha \in Q_n\left(Ran\left(T\right)\right)^{\perp}$, and so $Ran\left(T\right)^{\perp} \subseteq Q_n\left(Ran\left(T\right)\right)^{\perp}$. Consequently $Q_n\left(Ran\left(T\right)\right) \subseteq Ran\left(T\right)$. Suppose now that $Q_n\left(Ran\left(T\right)\right) \subseteq Ran\left(T\right)$ for each $n$. We want to show that $\left\{\xi_j\right\}_{j \in J}$ is a Riesz basis. Suppose that $\sum_{j \in J} \alpha_j \xi_j=0$, where $\alpha_j \in \mathcal{A}$. Fix an $n \in J$, then $Q_n T \xi \in Ran\left(T\right)$, so there exists $\eta_n \in \mathcal{X}$ such that $T \eta_n=Q_n T \xi$. Therefore we get
	\[
	\left\langle \eta_n, \xi_j\right\rangle= \begin{cases}\left\langle \xi, \xi_n\right\rangle & \text{if } j=n, \\ 0 & \text{if } j \neq n .\end{cases}
	\]
	Now for any $\xi \in \mathcal{X}$ we have
	\[
	\left\langle \xi, \alpha_n \xi_n\right\rangle=\left\langle \xi, \xi_n\right\rangle \alpha_n^*=\sum_{j \in J}\left\langle \eta_n, \xi_j\right\rangle \alpha_j^*=\sum_{j \in J}\left\langle \eta_n, \alpha_j \xi_j\right\rangle=\left\langle \eta_n, \sum_{j \ in J} \alpha_j \xi_j\right\rangle=0,
	\]
	which implies that $\alpha_n \xi_n=0$.
\end{proof}
\begin{theorem}\label{thrm2.22}
	Consider a frame $\left\{\xi_j\right\}_{j \in J}$ of a Hilbert $\mathcal{A}$-module $\mathcal{X}$, finitely or countably generated over a unital pro-$C^*$-algebra $\mathcal{A}$. This frame can be characterized as a Riesz basis under the following conditions: Firstly, each $\xi_n \neq 0$, and secondly, for every $n \in J$, we have $Q_n\left(Ran\left(T\right)\right) \subseteq Ran\left(T\right)$. Herein, $T$ represents the analysis operator associated with $\left\{\xi_j\right\}_{j \in J}$.
\end{theorem}

\begin{proof}
	Start by assuming that our frame $\left\{\xi_j\right\}_{j \in J}$ is indeed a Riesz basis. Given a sequence $\alpha=\left\{\alpha_j\right\}_{j \in J}$ within $l^2(\mathcal{A})$, if their linear combination with the frame elements results in zero, i.e., $\sum_{j \in J} \alpha_j \xi_j=0$, it directly implies that every term of the form $\alpha_j \xi_j$ is zero.
	
	For a sequence $\alpha$ that is orthogonal to $Ran\left(T\right)$ and any element $\xi$ from $\mathcal{X}$, one can write:
	\[
	0 = \sum_{j \in J}\left\langle \xi, \xi_j\right\rangle \alpha_j^*.
	\]
	From this, it is evident that for all $j$, $\alpha_j \xi_j=0$. Using this, we can deduce:
	\[
	\sum_{j=1}^n\left\langle \xi, \xi_j\right\rangle \alpha_j^* = 0.
	\]
	This result suggests that the sequence $\alpha$ is in the orthogonal space of $Q_n\left(Ran\left(T\right)\right)$. This leads to the conclusion that $Q_n\left(Ran\left(T\right)\right) \subseteq Ran\left(T\right)$.
	
	Conversely, given that the inclusion is true for all $n$, to demonstrate that $\left\{\xi_j\right\}_{j \in J}$ forms a Riesz basis, take an arbitrary linear combination of the form $\sum_{j \in J} \alpha_j \xi_j=0$, with coefficients $\alpha_j$ from $\mathcal{A}$. Picking any $n \in J$, and given that $Q_n T \xi$ lies in $Ran\left(T\right)$, there must be some $\eta_n$ in $\mathcal{X}$ for which:
	\[
	T \eta_n=Q_n T \xi.
	\]
	This relationship provides:
	\[
	\left\langle \eta_n, \xi_j\right\rangle = \left\langle \xi, \xi_n\right\rangle \text{ when } j=n \text{ and } 0 \text{ otherwise}.
	\]
	Upon evaluating any element $\xi$ in $\mathcal{X}$, we deduce:
	\[
	\left\langle \xi, \alpha_n \xi_n\right\rangle=0,
	\]
	leading to the conclusion that $\alpha_n \xi_n=0$.
\end{proof}

\begin{corollary}
	Let the set \(\left\{\xi_i\right\}_{i \in I}\) be a frame for \(\mathcal{X}\). This set qualifies as a Riesz basis if and only if:
	\begin{enumerate}
		\item For all \(i \in I\), we have \(\xi_i \neq 0\).
		\item Given that \(\sum_{i \in I} \alpha_i \xi_i = 0\) for a particular sequence \(\left\{\alpha_i ; i \in I\right\} \in l^2(\mathcal{A})\), it necessitates that \(\alpha_i \xi_i = 0\) for every \(i \in I\).
	\end{enumerate}
\end{corollary}
\begin{proof}

		Consult the detailed argument presented in Theorem \ref{thrm2.22}.

\end{proof}

\begin{definition}\cite{Haddad}
	Let's consider the sequence \(\Gamma = \left\{\Gamma_i \in Hom_\mathcal{A}^*\left(\mathcal{X}, \mathcal{\mathcal{X}}_i\right)\right\}_{i \in I}\). This sequence is defined as a \(g\)-frame for \(\mathcal{X}\) in relation to \(\left\{\mathcal{X}_i\right\}_{i \in I}\) if two positive constants \(C\) and \(D\) can be identified such that, for each element \(\xi \in \mathcal{X}\),
	\[
	C\langle \xi, \xi\rangle \leq \sum_{i \in I}\left\langle \Gamma_i \xi, \Gamma_i \xi\right\rangle \leq D\langle \xi, \xi\rangle .
	\]
	The constants \(C\) and \(D\) are termed the g-frame boundaries for \(\Gamma\). The \(g\)-frame is recognized as tight when \(C = D\), and it's described as Parseval when \(C\) and \(D\) are both 1. When solely the upper limit is necessary, then \(\Gamma\) is described as a \(g\)-Bessel sequence. Additionally, if for each index \(i \in I\), it holds that \(\mathcal{X}_i = \mathcal{X}\), it's described as a \(g\)-frame for \(\mathcal{X}\) with respect to \(\mathcal{X}\).
\end{definition}

 \begin{definition}
 	Consider a \(g\)-frame sequence \(\left\{\Gamma_i \in Hom_{\mathcal{A}}^*\left(\mathcal{X}, \mathcal{X}_i\right)\right\}_{i \in I}\) for \(\mathcal{X}\), in association with \(\left\{\mathcal{X}_i\right\}_{i \in I}\). This sequence is identified as a \(g\)-Riesz basis if the following conditions are met:
 	\begin{enumerate}
 			\item[(i)]  $\Gamma_i \neq 0$ for all $i \in I$.
 			\item[(ii)] Given any \(\mathcal{A}\)-linear combination expressed as \(\sum_{j \in S} \Gamma_j^* g_j\) which equates to zero, each individual term \(\Gamma_j^* g_j\) must also be zero. Here, \(g_j\) is an element of \(\mathcal{X}_i\) and \(S\) is a subset of \(I\).
 	\end{enumerate}
 \end{definition}

\begin{definition}
	Let $\left\{\Gamma_i \in Hom_{\mathcal{A}}^*\left(\mathcal{X}, \mathcal{X}_i\right): i \in I\right\}$ represent a $g$-frame associated with $\mathcal{X}$, and indexed by $\left\{\mathcal{X}_i\right\}_{i \in I}$. We term this $g$-frame as a $g$-Riesz basis if:
	\begin{enumerate}
		\item For every index $i \in I$, $\Gamma_i$ is distinct from zero.
		\item If any $\mathcal{A}$-linear mixture, denoted as $\sum_{j \in S} \Gamma_j^* g_j$, results in zero, it implies that each term, specifically $\Gamma_j^* g_j$, must also be null. Here, $g_j$ is an element of $\mathcal{X}_i$ and $S$ is included in $I$.
	\end{enumerate}
\end{definition}

\begin{theorem}\label{thrm2.10}
	Assume \(\left\{\Gamma_j\right\}_{j \in J}\) constitutes a \(g\)-frame for a Hilbert \(\mathcal{A}\)-module \(\mathcal{M}\) that is either finitely or countably generated, associated with \(\left\{\mathcal{N}_j\right\}\). Let the analysis operator related to \(\left\{\Gamma_j\right\}_{j \in J}\) be denoted by \(T_{\Gamma}: \mathcal{M} \longrightarrow \bigoplus_{j \in J} \mathcal{N}_j\). The set \(\left\{\Gamma_j\right\}_{j \in J}\) forms a \(g\)-Riesz basis if and only if \(\Gamma_n \neq 0\) for each \(n\), and the projection \(Q_n\left(\operatorname{Ran} T_{\Gamma}\right)\) is contained within \(\operatorname{Ran} T_{\Gamma}\).
\end{theorem}

\begin{proof}
	
	Start by assuming that the sequence $\left\{\Gamma_j\right\}_{j \in J}$ serves as a g-Riesz basis. Given $g = \left\{g_j\right\}_{j \in J} \in\left(\text{Ran} T_{\Gamma}\right)^{\perp}$, for every $\xi \in \mathcal{M}$, it follows that
	$$
	0 = \left\langle \left\{g_j\right\}_{j \in J}, T_{\Gamma} \xi\right\rangle = \left\langle T_{\Gamma}^*\left\{g_j\right\}_{j \in J}, \xi\right\rangle = \left\langle\sum_{j \in J} \Gamma_j^* g_j, \xi\right\rangle.
	$$
	
	From this, we deduce that $\sum_{j \in J} \Gamma_j^* g_j = 0$ and thus, $\Gamma_j^* g_j = 0$ for every $j \in J$.
	
	Now, for any $\xi \in \mathcal{M}$, consider:
	$$
	\left\langle\left\{g_j\right\}_{j \in J}, Q_n\left(T_{\Gamma} \xi\right)\right\rangle = \left\langle\left\{g_j\right\}_{j \in J}, Q_n\left(\left\{\Gamma_j \xi\right\}_{j \in J}\right)\right\rangle = \left\langle g_n, \Gamma_n \xi\right\rangle = \left\langle \Gamma_n^* g_n, \xi\right\rangle = 0.
	$$
	
	This implies $\left(\operatorname{Ran} T_{\Gamma}\right)^{\perp} \subseteq Q_n\left(\operatorname{Ran} T_{\Gamma}\right)^{\perp}$. Therefore, $Q_n\left(\operatorname{Ran} T_{\Gamma}\right) \subseteq \operatorname{Ran} T_{\Gamma}$.
	
	Moving forward, if we assume that $Q_n\left(\operatorname{Ran} T_{\Gamma}\right) \subseteq \operatorname{Ran} T_{\Gamma}$ for all $n$, and if $\sum_{j \in J} \Gamma_j^* g_j = 0$, then by fixing any $n \in J$, it's evident that $Q_n T_{\Gamma} \xi \in \operatorname{Ran} T_{\Gamma}$. Hence, there's an $\xi_n \in \mathcal{M}$ such that $Q_n T_{\Gamma} \xi = T_{\Gamma} \xi_n$.
	
	Now, for every $\xi \in \mathcal{M}$, it is true that:
	$$
	\left\langle \xi, \Gamma_n^* g_n\right\rangle = \left\langle \Gamma_n \xi, g_n\right\rangle = \sum_{j \in J}\left\langle \Gamma_j f_n, g_j\right\rangle = \left\langle f_n, \sum_{j \in J} \Gamma_j^* g_j\right\rangle = 0,
	$$
	leading to the conclusion that $\Gamma_j^* g_j = 0$ for every $n \in J$.
	
\end{proof}

\begin{corollary}
	Suppose that $\left\{\Gamma_i \in Hom_{\mathcal{A}}^*\left(\mathcal{X}, \mathcal{X}_i\right): i \in I\right\}$ constitutes a g-frame for $\mathcal{X}$ relative to  $\left\{\mathcal{X}_i\right\}_{i \in I}$. This set is then recognized as a $g$-Riesz basis under the conditions that:
	\begin{enumerate}
		\item[(i)] Every $\Gamma_i \neq 0$  for all $i$ in the index set $I$.
		\item[(ii)] For a given sequence $\left\{g_i\right\}_{i \in I}$ belonging to $l^2\left(\left\{\mathcal{X}_i\right\}_{i \in I}\right)$ such that $\sum_{i \in I} \Gamma_i^* g_i = 0$, it is imperative that $\Gamma_i^* g_i = 0$ for all individual indices $i$ in $I$.
	\end{enumerate}
\end{corollary}

\begin{proof}
Refer to the demonstration provided in Theorem \ref{thrm2.10}.	
\end{proof}

\begin{definition}

	Consider a sequence $\left\{\Gamma_i \in Hom_{\mathcal{A}}^*\left(\mathcal{X}, \mathcal{X}_i\right): i \in I\right\}$. This sequence is referred to as a $g$-orthonormal basis for $\mathcal{X}$ relative to $\{ \mathcal{X}_i\}_{i \in I}$ provided
	\begin{equation}\label{eq2.3}
	\left\langle\Gamma_i^* g_i, \Gamma_j^* g_j\right\rangle = \delta_{i, j}\left\langle g_i, g_j\right\rangle \text{ whenever } i, j \in I, g_i \in \mathcal{X}_i, \text{ and } g_j \in \mathcal{X}_j,
	\end{equation}
	 and
	$$
	\sum_{i \in I}\left\langle\Gamma_i \xi, \Gamma_i \xi\right\rangle=\langle \xi, \xi\rangle, \text { for all } \xi \in \mathcal{X} .
	$$
\end{definition}
\begin{definition}\cite{Rouma}
	
	Given $K \in Hom_{\mathcal{A}}^{\ast}(\mathcal{X})$ and a sequence in $\mathcal{X}$ represented by $\{ \xi_{i} \}_{i=0}^{\infty}$, this sequence is termed a $K$-frame for $\mathcal{X}$ if and only if there are two distinct positive constants, denoted as $A$ and $B$, such that
	$$A \langle K^{\ast} \xi, K^{\ast}\xi \rangle_{\mathcal{X}} \leq \sum_{i \in I}\left\langle \xi, \xi_{i}\right\rangle \left\langle \xi_{i}, \xi \right\rangle \leq B \langle  \xi, \xi \rangle_{\mathcal{X}} \; \; \forall \xi \in \mathcal{X} $$
The constants $A$ and $B$ are named the lower and upper bounds of the $K$-frame $\{ \xi_{i} \}_{i \in I}$ respectively.
\end{definition}
For a given $K$-Frame represented by $\{ \xi_i \}_{i \in I}$, the analysis operator is defined by:

$$
T: \mathcal{X} \rightarrow l^2(\mathcal{A}), \quad T(\xi)=\left\{\left\langle \xi, \xi_i\right\rangle\right\}_{i \in I} .
$$
Upon examining this, one can ascertain that its adjoint operator is expressed as:
\begin{eqnarray}\label{eq2.5}
T^*: l^2(\mathcal{A}) \rightarrow \mathcal{X}, \quad T^*\left(\left\{\alpha_i\right\}_{i \in I}\right)=\sum_{i \in I} \alpha_i \xi_i .
\end{eqnarray}
This adjoint operator, $T^*$, is termed the synthesis operator associated with $\{ \xi_i\}_{i \in I}$.

When combining $T^*$ and $T$, the result is the frame operator $S$ for the $K$-frames, which can be articulated as:
$$
S: \mathcal{X} \rightarrow \mathcal{X}, \quad S(\xi)=T^* T(\xi)=\sum_{i \in I}\left\langle \xi, \xi_i\right\rangle \xi_i
$$
\begin{proposition}\cite{Rouma}\label{propo2.12}
The sequence \(\left\{\xi_i\right\}_{i \in I}\) is recognized as a Bessel sequence for \(\mathcal{X}\) bounded by \(D\) if, and only if, the operator \(T^*\) given by \ref{eq2.5} is both well-defined and bounded with \(\bar{p}_\mathcal{X}(T^*) \leq \sqrt{D}\).

\end{proposition}

\begin{definition}
	Given $K \in Hom_{\mathcal{A}}^{*}(\mathcal{X})$, we denote by $\Gamma$ the sequence $\left\{\Gamma_{i} \in Hom_{A}^{*}\left(\mathcal{X}, \mathcal{X}_{i}\right)\right\}_{i \in I}$. This sequence is termed a $K$-g-frame for $\mathcal{X}$ relative to $\left\{\mathcal{X}_{i}\right\}_{i \in I}$ provided there are two positive constants, $C$ and $D$, for which the following holds for every $\xi \in \mathcal{X}$:
	$$
	C\langle K^{*} \xi, K^{*}\xi \rangle \leq \sum_{i \in I}\left\langle\Gamma_{i} \xi, \Gamma_{i} \xi \right\rangle \leq D \langle \xi, \xi \rangle
	$$
In this context, the constants $C$ and $D$ represent the lower and upper boundaries of the $K$-g-frame. The $K$-g-frame is labeled as tight when $C = D$, and it takes on the name Parseval if both $C$ and $D$ are equal to 1.

\end{definition}

The analysis operator corresponding to a \(K\)-g-Frame \(\left\{\Gamma_i \in Hom_{\mathcal{A}}^*\left(\mathcal{X}, \mathcal{X}_i\right): i \in I\right\}\) can be defined as:
$$
T_{\Gamma}: \mathcal{X} \rightarrow l^2\left(\left\{\mathcal{X}_i\right\}_{i \in I}\right), \quad T_{\Gamma}(\xi)=\left\{\Gamma_i \xi\right\}_{i \in I} .
$$

Given the above definition, the corresponding adjoint operator of \(T_{\Gamma}\) can be represented as:
\begin{equation}\label{eq2.7}
T_{\Gamma}^*: l^2\left(\left\{\mathcal{X}_i\right\}_{i \in I}\right) \rightarrow \mathcal{X}, \quad T_{\Gamma}^*\left(\left\{g_i\right\}_{i \in I}\right)=\sum_{i \in I} \Gamma_i^* g_i .
\end{equation}
In this context, \(T_{\Gamma}^*\) is identified as the synthesis operator for the set \(\left\{\Gamma_i\right\}_{i \in I}\). 
When combining operations from \(T_{\Gamma}^*\) and \(T_{\Gamma}\), the frame operator \(S_{\Gamma}\) pertaining to the \(K\)-g-frames emerges as:
$$
S_{\Gamma}: \mathcal{X} \rightarrow \mathcal{X}, \quad S_{\Gamma}(\xi)=T_{\Gamma}^* T_{\Gamma}(\xi)=\sum_{i \in I} \Gamma_i^* \Gamma_i \xi
$$

\begin{proposition}\label{propo2.14}
	Let $\left\{ \Gamma_i \in Hom_{\mathcal{A}}^* \left(\mathcal{X}, \mathcal{X}_i\right) : i \in I \right\}$ be a set. This set is a $g$-Bessel sequence for $\mathcal{X}$ with a bound of $D$ if and only if the operator $T_{\Gamma}^*$, referenced in \ref{eq2.7}, is a properly defined bounded operator with the property that $\bar{p}_{\mathcal{X}}(T_{\Gamma}^*) \leq \sqrt{D}$.
\end{proposition}

\begin{proof}
The proof of this proposition proceeds analogously to that of Proposition \ref{propo2.12}	
\end{proof}

\begin{lemma}\cite{Joita}\label{lem2.15}
	Given Hilbert pro-$C^*$-modules $\mathcal{X}_1$ and $\mathcal{X}_2$, and $P \in Hom_{\mathcal{A}}^*\left(\mathcal{X}_1, \mathcal{X}_2\right)$:
	\begin{enumerate}
		\item[(i)] $Ran(P)^{\perp}=Ker\left(T^*\right)$ and $\overline{Ran(P)} \subseteq Ran(P)^{\perp \perp}=Ker\left(P^*\right)^{\perp}$.
		\item[(ii)] $Ran(P)$ is closed iff $Ran\left(P^*\right)$ is closed; if true, $Ran(P)=Ker\left(P^*\right)^{\perp}$ and $Ran\left(P^*\right)=Ker(P)^{\perp}$.
	\end{enumerate}
\end{lemma}

\section{Main result}
\subsection{$K$-Riesz bases  in Hilbert Pro-$C^*$-Modules}
\begin{definition}

	Suppose $K$ belongs to $Hom_{\mathcal{A}}^*(\mathcal{X})$. A sequence, denoted as $\{ \xi_i \}_{i \in I}$ within the space $\mathcal{X}$, is recognized as a $K$-Riesz basis under the following conditions:
	\begin{enumerate}
		\item The group of elements given by  $\left\{\xi \in \mathcal{X}:\left\langle \xi, \xi_i\right\rangle=0\right.$, for all $\left.i \in I\right\} $
		is included in the kernel of the adjoint operator $K^*$.
		
		\item There are specific positive constants, noted as $C$ and $D$, such that for any sequence $\left\{\alpha_i\right\}_{i \in I}$ in $l^2(\mathcal{A})$, the following holds:
		$$
	C \sum_{i \in I} \alpha_i \alpha_i^* \leq\left\langle\sum_{i \in I} \alpha_i \xi_i, \sum_{i \in I} \alpha_i \xi_i\right\rangle \leq D \sum_{i \in I} \alpha_i \alpha_i^* .
	$$
	\end{enumerate}

\end{definition}
	
	\begin{theorem}
		Let's consider a sequence $\left\{\xi_i\right\}_{i \in I}$ that serves as a $K$-frame for $\mathcal{X}$. This sequence $\left\{\xi_i\right\}_{i \in I}$ is recognized as a $K$-Riesz basis for $\mathcal{X}$ if, and only if, the synthesis operator $T^*$ associated with this $K$-frame is lower bounded.
	\end{theorem}

	 \begin{proof}
	 	$\quad(\Rightarrow)$ Let's begin by hypothesizing that $\left\{\xi_i\right\}_{i \in I}$ represents a $K$-Riesz basis for $\mathcal{X}$. This implies:
	 	\begin{align*}
	 	\left\{\xi \in \mathcal{X}:\left\langle \xi, \xi_i\right\rangle=0, \forall i \in I\right\} &\subset Ker\left(K^*\right), \\
	 	C \sum_{i \in I} \alpha_i \alpha_i^* &\leq\left\langle\sum_{i \in I} \alpha_i \xi_i, \sum_{i \in I} \alpha_i \xi_i\right\rangle \leq D \sum_{i \in I} \alpha_i \alpha_i^*.
	 	\end{align*}
	 	From this, it follows that:
	 	\[
	 	C\left\langle\left\{\alpha_i\right\}_{i \in I},\left\{\alpha_i\right\}_{i \in I}\right\rangle \leq\left\langle T^*\left(\left\{\alpha_i\right\}_{i \in I}\right), T^*\left(\left\{\alpha_i\right\}_{i \in I}\right)\right\rangle \leq D\left\langle\left\{\alpha_i\right\}_{i \in I},\left\{\alpha_i\right\}_{i \in I}\right\rangle.
	 	\]
	 	Consequently, $T^*$ is bounded below.
	 	
	 	$\quad(\Leftarrow)$ Given that $\left\{\xi_i\right\}_{i \in I}$ acts as a $K$-frame, we can identify constants $C, D>0$ such that:
	 	\[
	 	C\left\langle K^* \xi, K^* \xi \right\rangle \leq \sum_{i \in I}\left\langle \xi, \xi_i\right\rangle\left\langle \xi_i, \xi\right\rangle \leq D \langle \xi, \xi\rangle.
	 	\]
	 	If $\left\langle \xi, \xi_i\right\rangle=0$ for all $i \in I$, then we deduce $C\left\langle K^*\xi, K^* \xi\right\rangle=0$ and consequently $\xi \in Ker\left(K^*\right)$. Therefore:
	 	\[
	 	\left\{\xi \in \mathcal{X}:\left\langle \xi, \xi_i\right\rangle=0\right., \forall i \in I\left.\right\} \subset Ker\left(K^*\right).
	 	\]
	 	Based on the fact that $\left\{\xi_i\right\}_{i \in I}$ is a Bessel sequence and referring to Proposition \ref{propo2.12}, we deduce:
	 	\[
	 	\bar{p}_\mathcal{X}(T^*) \leq \sqrt{D}.
	 	\]
	 	Thus:
	 	\begin{align*}
	 	\left\langle\sum_{i \in I} \alpha_i \xi_i, \sum_{i \in I} \alpha_i \xi_i\right\rangle &\leq D \sum_{i \in I} \alpha_i \alpha_i^*.
	 	\end{align*}
	 	Since $T^*$ is lower bounded, an existence of a positive constant $m$ implies:
	 	\begin{align*}
	 	m \sum_{i \in I} \alpha_i \alpha_i^* &\leq\left\langle\sum_{i \in I} \alpha_i \xi_i, \sum_{i \in I} \alpha_i \xi_i\right\rangle.
	 	\end{align*}
	 	Thus, we conclude:
	 	\[
	 	m \sum_{i \in I} \alpha_i \alpha_i^* \leq\left\langle\sum_{i \in I} \alpha_i \xi_i, \sum_{i \in I} \alpha_i \xi_i\right\rangle \leq D \sum_{i \in I} \alpha_i \alpha_i^*.
	 	\]
	 	From this, $\left\{\xi_i\right\}_{i \in I}$ qualifies as a $K$-Riesz basis for $\mathcal{X}$.
	 \end{proof}

\begin{theorem}\label{thrm3.3}

		Consider the sequence $\left\{\xi_i\right\}_{i \in I}$. It qualifies as a $K$-Riesz basis for $\mathcal{X}$ if and only if there's an operator $\Phi$, that's bounded from below and belongs to $Hom_{\mathcal{A}}^*(\mathcal{X})$, meeting the criteria: $Ran(K)$ is contained within $Ran(\Phi)$, and for each $i \in I$, $\Phi u_i =\xi_i$. Here, $\left\{u_i\right\}_{i \in I}$ is identified as the standard orthonormal basis for $\mathcal{X}$.

\end{theorem}

\begin{proof}
	
	Assume that the sequence $\left\{\xi_i\right\}_{i \in I}$ forms a $K$-Riesz basis for $\mathcal{X}$. From this assumption, we have:
	\begin{equation}\label{eq3.1}
	\left\{\xi \in \mathcal{X} : \left\langle \xi, \xi_i \right\rangle = 0 \text{, for all } i \in I\right\} \subset Ker\left(K^*\right).
	\end{equation}
	There also exist constants $C, D > 0$ such that for every $\left\{\alpha_i\right\}_{i \in I} \in l^2(\mathcal{A})$:
	\begin{equation}\label{eq3.2}
	C \sum_{i \in I} \alpha_i \alpha_i^* \leq \left\langle \sum_{i \in I} \alpha_i \xi_i, \sum_{i \in I} \alpha_i \xi_i \right\rangle \leq D \sum_{i \in I} \alpha_i \alpha_i^*.
	\end{equation}
	Now, define the operator $\Phi: \mathcal{X} \rightarrow \mathcal{X}$ as:
	$$
	\Phi(\eta) = \sum_{i \in I} \left\langle \eta, u_i \right\rangle \xi_i.
	$$
	For every $\eta \in \mathcal{X}$, we know that $\{ \langle \eta, u_i\rangle \}_{i \in I} \in l^2(\mathcal{A})$. Utilizing (\ref{eq3.2}), we deduce:
	$$
	\begin{aligned}
	\langle \Phi(\eta), \Phi(\eta) \rangle & = \left\langle \sum_{i \in I} \left\langle \eta, u_i \right\rangle \xi_i, \sum_{i \in I} \left\langle \eta, u_i \right\rangle \xi_i \right\rangle \\
	& \leq D \sum_{i \in I} \left\langle \eta, u_i \right\rangle \left\langle u_i, \eta \right\rangle \\
	& = D \langle \eta, \eta \rangle.
	\end{aligned}
	$$
	This confirms that $\Phi$ is a bounded operator and that $\Phi u_i = \xi_i$ for all $i \in I$. Additionally, from \ref{eq3.2}, we can infer:
	\begin{equation}\label{eq3.3}
	C\langle \eta, \eta \rangle \leq \langle \Phi(\eta), \Phi(\eta) \rangle \leq D \langle \eta, \eta \rangle, \quad \forall \eta \in \mathcal{X}.
	\end{equation}
	Equation (\ref{eq3.3}) implies that $\Phi$ is bounded below. 
	
	Continuing, consider the adjoint:
	$$
	\Phi^*: \mathcal{X} \rightarrow \mathcal{X}, \quad \Phi^*(\xi) = \sum_{i \in I} \left\langle \xi, \xi_i \right\rangle u_i.
	$$
	If $\xi \in \mathcal{X}$ and $\Phi^*(\xi) = 0$, then for every $i \in I$, $\left\langle \xi, \xi_i \right\rangle = 0$. Hence, according to (\ref{eq3.1}), $\xi \in Ker(K^*)$. This implies that $Ker(\Phi^*) \subset Ker(K^*)$, leading to the conclusion $Ker(K^*)^\perp \subset Ker(\Phi^*)^\perp$. Utilizing Lemma \ref{lem2.15}, it follows that $Ran(K) \subset \overline{Ran(K)} \subset \overline{Ran(\Phi)} = Ran(\Phi)$.
	
	For the converse, let $\Phi \in Hom_{\mathcal{A}}^*(\mathcal{X})$ be a bounded below operator satisfying $\Phi u_i = \xi_i$ for all $i \in I$ and $Ran(K) \subset Ran(\Phi)$. Then, for every sequence $\left\{\alpha_i\right\}_{i \in I} \in l^2(\mathcal{A})$, we deduce:
	$$
	\begin{aligned}
	\left\langle \sum_{i \in I} \alpha_i \xi_i, \sum_{i \in I} \alpha_i \xi_i \right\rangle & = \left\langle \sum_{i \in I} \alpha_i \Phi u_i, \sum_{i \in I} \alpha_i \Phi u_i \right\rangle \\
	& \leq \|\Phi\|_{\infty}^2 \sum_{i \in I} \alpha_i \alpha_i^*.
	\end{aligned}
	$$
	Given a positive constant $m$, we further have:
	$$
	\begin{aligned}
	m \sum_{i \in I} \alpha_i \alpha_i^* & = m \left\langle \sum_{i \in I} \alpha_i u_i, \sum_{i \in I} \alpha_i u_i \right\rangle \\
	& \leq \left\langle \Phi \left( \sum_{i \in I} \alpha_i u_i \right), \Phi \left( \sum_{i \in I} \alpha_i u_i \right) \right\rangle \\
	& = \left\langle \sum_{i \in I} \alpha_i \xi_i, \sum_{i \in I} \alpha_i \xi_i \right\rangle.
	\end{aligned}
	$$
	This demonstrates that (\ref{eq3.2}) holds for the sequence $\left\{\xi_i\right\}_{i \in I}$ with the constants $m$ and $\|\Phi\|_{\infty}^2$. Furthermore, for any $\xi \in \mathcal{X}$ such that $\left\langle \xi, \xi_i \right\rangle = 0$ for all $i \in I$, we conclude:
	$$
	0 = \left\langle \xi, \Phi u_i \right\rangle = \left\langle \Phi^* \xi, u_i \right\rangle,
	$$
	leading to the result $\Phi^* \xi = 0$. As $Ran(K) \subset Ran(\Phi)$ and by employing Lemma \ref{lem2.15}, we can infer $Ker(\Phi^*) \subset Ker(K^*)$. This confirms the relationship:
	$$
	\left\{\xi \in \mathcal{X} : \left\langle \xi, \xi_i \right\rangle = 0 \text{, for all } i \in I\right\} \subset Ker\left(K^*\right).
	$$
	With this result and the earlier derived bounds, we can conclude that the sequence $\left\{\xi_i\right\}_{i \in I}$ forms a $K$-Riesz basis for $\mathcal{X}$.
	
\end{proof}

\subsection{ $K$-$g$-Riesz bases in Hilbert Pro-$C^*$-Modules}

\begin{definition}
	Let \( K \) be an operator within \( Hom_{\mathcal{A}}^*(\mathcal{X}) \). We say that the sequence 
	\[
	\left\{ \Gamma_i \mid \Gamma_i \in Hom_{\mathcal{A}}^*\left(\mathcal{X}, \mathcal{X}_i\right) \right\}_{i \in I}
	\]
	in \( \mathcal{X} \) acts as a \( K \)-\( g \)-Riesz basis if and only if:
	
	\begin{itemize}
		\item[(a)]$\left\{\xi \in \mathcal{X}: \Gamma_i \xi=0\right.$, for all $\left.i \in I\right\} \subset Ker\left(K^*\right)$.
		\item[(b)] There exist positive constants \( C \) and \( D \) such that, for any sequence \( \{g_i\}_{i \in I} \) in \( l^2\left(\{\mathcal{X}_i\}_{i \in I}\right) \), the following inequality holds:
		\[
		C \sum_{i \in I}\left\langle g_i, g_i\right\rangle \leq \left\langle\sum_{i \in I} \Gamma_i^* g_i, \sum_{i \in I} \Gamma_i^* g_i\right\rangle \leq D \sum_{i \in I}\left\langle g_i, g_i\right\rangle.
		\]
	\end{itemize}
\end{definition}

\begin{example}
	Consider the standard orthonormal set \( \{u_i\}_{i=0}^{\infty} \) within the space \( \mathcal{X} \). For each \( i \) in the natural numbers, introduce the operator \( \Gamma_i: \mathcal{X} \to \mathcal{A} \) such that \( \Gamma_i \xi = \langle \xi, u_i \rangle \). Simultaneously, describe the operator \( K: \mathcal{X} \to \mathcal{X} \) as
	\[
	K \xi = \sum_{i=1}^{\infty} \langle \xi, u_i \rangle u_i.
	\]
	It's evident from the above that \( K^* \xi = \sum_{i=1}^{\infty} \langle \xi, u_i \rangle u_i \) and for any \( \alpha \) in \( \mathcal{A} \), \( \Gamma_i^* \alpha = \alpha u_i \). For any given \( \xi \) in \( \mathcal{X} \), the relationship holds as:
	\[
	\langle K^* \xi, K^* \xi \rangle = \sum_{i=1}^{\infty} \langle \xi, u_i \rangle \langle u_i, \xi \rangle = \sum_{i=1}^{\infty} \langle \Gamma_i \xi, \Gamma_i \xi \rangle.
	\]
	Should \( \Gamma_i \xi = 0 \) for all \( i \) in the natural numbers, it follows that \( K^* \xi = 0 \). This leads to the conclusion that any \( \xi \) in \( \mathcal{X} \) for which \( \Gamma_i \xi = 0 \) resides within the kernel of \( K^* \). Moreover, for every sequence \( \{\alpha_i\}_{i \in \mathbf{N}} \) in \( l^2(\mathcal{A}) \), the equation stands as:
	\[
	\sum_{i \in \mathbf{N}} \langle \alpha_i, \alpha_i \rangle = \sum_{i \in \mathbf{N}} \alpha_i \alpha_i^*.
	\]
	As a result, \( \{\Gamma_i\}_{i \in \mathbf{N}} \) serves as a \( K \)-g-Riesz basis in \( \mathcal{X} \) relative to \( \mathcal{A} \).
\end{example}

\begin{theorem}\label{thrm3.6}
	Suppose \( \{\Gamma_i \in Hom_{\mathcal{A}}^*(\mathcal{X}, \mathcal{X}_i) : i \in I\} \) constitutes a \( K \)-g-frame in \( \mathcal{X} \) related to the set \( \{\mathcal{X}_i\}_{i \in I} \). Then, \( \{\Gamma_i\}_{i \in I} \) establishes a \( K \)-g-Riesz basis within \( \mathcal{X} \) in association with \( \{\mathcal{X}_i\}_{i \in I} \) if and only if, the synthesis operator \( T_{\Gamma}^* \) associated with the \( K \)-g-frame \( \{\Gamma_i\}_{i \in I} \) possesses a bounded lower limit.
\end{theorem}

\begin{proof}
	$(\Rightarrow)$, Suppose that the set \( \{\Gamma_i\}_{i \in I} \) acts as a \( K \)-g-Riesz basis for \( \mathcal{X} \) associated with \( \{\mathcal{X}_i\}_{i \in I} \). Consequently, 
	$$
	\{\xi \in \mathcal{X} | \Gamma_i \xi=0 \text{ for every } i \in I\} \subseteq Ker(K^*),
	$$
	and certain constants \( C, D > 0 \) can be found ensuring that for any \( \{g_i\}_{i \in I} \) from \( l^2(\{\mathcal{X}_i\}_{i \in I}) \),
	$$
	C \sum_{i \in I}\langle g_i, g_i\rangle \leq \langle\sum_{i \in I} \Gamma_i^* g_i, \sum_{i \in I} \Gamma_i^* g_i\rangle \leq D \sum_{i \in I}\langle g_i, g_i\rangle.
	$$
	From this, it follows that
	$$
	C\langle\{g_i\}_{i \in I},\{g_i\}_{i \in I}\rangle \leq \langle T_{\Gamma}^*\{g_i\}_{i \in I}, T_{\Gamma}^*\{g_i\}_{i \in I}\rangle \leq D\langle\{g_i\}_{i \in I},\{g_i\}_{i \in I}\rangle,
	$$
	which implies the operator \( T_{\Gamma}^* \) has a bounded minimum value.
	
	$(\Leftarrow)$, With the given \( \{\Gamma_i\}_{i \in I} \) being a \( K \)-g-Frame, constants \( C, D > 0 \) are found such that for every \( \xi \) in \( \mathcal{X} \),
	$$
	C\langle K^* \xi, K^* \xi\rangle \leq \sum_{i \in I}\langle \Gamma_i \xi, \Gamma_i \xi \rangle \leq D\langle \xi, \xi\rangle.
	$$
	Thus, if all \( \Gamma_i \xi \) equals zero for each \( i \) in \( I \), then \( C\langle K^* \xi, K^* \xi\rangle = 0 \), meaning \( \xi \) is an element of \( Ker(K^*) \). Consequently, \( \{\xi \in \mathcal{X} | \Gamma_i \xi = 0\} \subseteq Ker(K^*) \). Given that \( \{\Gamma_i\}_{i \in I} \) is recognized as a g-Bessel sequence, as per Proposition \ref{propo2.14} \( T_{\Gamma}^* \) is bounded, and a positive \( D \) exists such that \( \bar{p}_{\mathcal{X}}(T_{\Gamma}^*) \leq \sqrt{D} \). This implies that for all \( \{g_i\}_{i \in I} \) in \( l^2(\{\mathcal{X}_i\}_{i \in I}) \),
	$$
	\langle\sum_{i \in I} \Gamma_i^* g_i, \sum_{i \in I} \Gamma_i^* g_i\rangle \leq D\langle\{g_i\}_{i \in I},\{g_i\}_{i \in I}\rangle = D \sum_{i \in I}\langle g_i, g_i\rangle.
	$$
	Since \( T_{\Gamma}^* \) is bounded below, a positive value \( m \) can be determined, ensuring that
	$$
	m\langle\{g_i\}_{i \in I},\{g_i\}_{i \in I}\rangle \leq \langle T_{\Gamma}^*\{g_i\}_{i \in I}, T_{\Gamma}^*\{g_i\}_{i \in I}\rangle,
	$$
	for every \( \{g_i\}_{i \in I} \) in \( l^2(\{\mathcal{X}_i\}_{i \in I}) \). As a result,
	$$
	m \sum_{i \in I}\langle g_i, g_i\rangle \leq \langle\sum_{i \in I} \Gamma_i^* g_i, \sum_{i \in I} \Gamma_i^* g_i\rangle \leq D \sum_{i \in I}\langle g_i, g_i\rangle,
	$$
	confirming that \( \{\Gamma_i\}_{i \in I} \) is indeed a \( K \)-g-Riesz basis for \( \mathcal{X} \) related to \( \{\mathcal{X}_i\}_{i \in I} \).
\end{proof}

\begin{example}
	Consider a standard orthonormal basis $\left\{u_i\right\}_{i=0}^{\infty}$ for the space $\mathcal{X}$. For each $i$ in the set of natural numbers, we set $\mathcal{X}_i$ to be equivalent to $\mathcal{A}^2$. We can then introduce bounded operators, denoted $\Gamma_i: \mathcal{X} \rightarrow \mathcal{A}^2$, formulated as $\Gamma_i \xi = \left(\left\langle \xi, u_i\right\rangle,\left\langle \xi, u_{i+1}\right\rangle\right)$. Simultaneously, we describe another operator $K: \mathcal{X} \rightarrow \mathcal{X}$ by the relationship $K \xi = \sum_{i=0}^{\infty}\left\langle \xi, u_i\right\rangle u_{i+1} + \sum_{i=0}^{\infty}\left\langle \xi, u_{i+1}\right\rangle u_i$.
	
	From this, it emerges that the adjoint operator of $K$ is given by:
	$$
	K^* \xi=\sum_{i=0}^{\infty}\left\langle \xi, u_i\right\rangle u_{i+1}+\sum_{i=0}^{\infty}\left\langle \xi, u_{i+1}\right\rangle u_i, \text { for all } \xi \in \mathcal{X}
	$$
	Moreover, the inner product of $\Gamma_i \xi$ with itself is expressed as:
	$$
	\begin{aligned}
	\left\langle\Gamma_i \xi, \Gamma_i \xi\right\rangle & =\left\langle\left(\left\langle \xi, u_i\right\rangle,\left\langle \xi, u_{i+1}\right\rangle\right),\left(\left\langle \xi, u_i\right\rangle,\left\langle \xi, u_{i+1}\right\rangle\right)\right\rangle \\
	& =\left\langle \xi, u_i\right\rangle\left\langle u_i, \xi\right\rangle+\left\langle \xi, u_{i+1}\right\rangle\left\langle u_{i+1}, \xi\right\rangle, \text { for all } \xi \in \mathcal{X} \text { and } i \in \mathbf{N}.
	\end{aligned}
	$$
	
	For any vector $\xi$ in $\mathcal{X}$:
	$$
	\begin{aligned}
	\left\langle K^* \xi, K^* \xi\right\rangle & =\left\langle\sum_{i=0}^{\infty}\left\langle \xi, u_i\right\rangle u_{i+1}+\sum_{i=0}^{\infty}\left\langle \xi, u_{i+1}\right\rangle u_i, \sum_{i=0}^{\infty}\left\langle \xi, u_i\right\rangle u_{i+1}\right. \\
	& +\sum_{i=0}^{\infty}\left\langle \xi, u_{i+1}\right\rangle u_i \\
	& \leq \sum_{i=0}^{\infty}\left\langle \xi, u_i\right\rangle\left\langle u_i, \xi\right\rangle+\sum_{i=0}^{\infty}\left\langle \xi, u_{i+1}\right\rangle\left\langle u_{i+1}, \xi\right\rangle \\
	& \leq \sum_{i=0}^{\infty}\left\langle\Gamma_i \xi, \Gamma_i \xi\right\rangle .
	\end{aligned}
	$$
	
	It is clear that the sequence $\left\{\Gamma_i\right\}_{i \in \mathbf{N}}$ acts as a $K$-g-frame for $\mathcal{X}$ with respect to $\mathcal{A}^2$. Further, if for any chosen vector $\xi$ from $\mathcal{X}$ we have $\Gamma_i \xi = 0$ for all $i$ in the natural numbers, then it directly leads to $\left\langle K^* \xi, K^* \xi\right\rangle = 0$. This indicates that  $\left\{\xi \in \mathcal{X}: \Gamma_i \xi=0\text{ for all } i \in \mathbf{N}\right\}$ is a subset of $Ker(K^*)$.
	
	Delving further, for every index $i$ in the set of natural numbers, we have:
	\begin{equation}\label{eq3.5}
	\Gamma_i^*(\alpha, \beta)=\alpha u_i+\beta u_{i+1}, \text { for all }(\alpha, \beta) \in \mathcal{A}^2.
	\end{equation}
	
	Considering the sequence $\left\{\left(\alpha_i, \beta_i\right)\right\}_{i=0}^{\infty}$ which is defined as:
	$\left(\alpha_0, \beta_0\right)=(0,-1), \quad\left(\alpha_1, \beta_1\right)=(1,0), \quad\left(\alpha_i, \beta_i\right)=(0,0)$, for all $i \geq 2$.
	
	We observe that:
	$$
	\sum_{i=0}^{\infty} \Gamma_i^*\left(\alpha_i, \beta_i\right)=\Gamma_0^*\left(\alpha_0, \beta_0\right)+\Gamma_1^*\left(\alpha_1, \beta_1\right)=\Gamma_0^*(0,-1)+\Gamma_1^*(1,0)=-u_2+u_2=0.$$
	\end{example}

\begin{lemma}
	Consider the sequence \(\left\{\Gamma_i \in Hom_{\mathcal{A}}^*\left(\mathcal{X}, \mathcal{X}_i\right)\right\}_{i \in I}\). Assume for every index \(i \in I\), the set \(\left\{w_{i, j}: j \in J_i\right\}\) forms a standard orthonormal basis for \(\mathcal{X}_i\), with \(J_i\) being a subset of \(\mathbf{Z}\). Given the relationship 
	\begin{equation}\label{eq3.6}
	w_{i, j}=\Gamma_i^* u_{i, j} ; \text{ for all } i \in I \text{ and } j \in J_i ,
	\end{equation}
	we refer to the sequence \(\left\{w_{i, j}: i \in I, j \in J_i\right\}\) as the one derived from \(\left\{\Gamma_i\right\}_{i \in I}\) in relation to \(\left\{u_{i, j}: i \in I, j \in J_i\right\}\). Furthermore, for every index \(i \in I\), the subsequent relationships are valid:

\begin{equation}
	\Gamma_i \xi  =\sum_{j \in J_i}\left\langle \xi, w_{i, j}\right\rangle u_{i, j} , \text{ for any } \xi \in \mathcal{X},
\end{equation} 
\begin{equation}\label{eq3.8}
	\Gamma_i^* g_i  =\sum_{j \in J_i}\left\langle g_i, u_{i, j}\right\rangle w_{i, j} , \text{ for any } g_i \in \mathcal{X}_i .
\end{equation}
\end{lemma}

\begin{proof}
	Given that for every \(i \in I\), the collection \(\left\{u_{i, j}: j \in J_i\right\}\) forms a standard orthonormal basis for \(\mathcal{X}_i\):
	
	1. For any chosen \(\xi \in \mathcal{X}\) and for all \(i \in I\), it follows that:
	\[
	\Gamma_i \xi = \sum_{j \in J_i} \left\langle \Gamma_i \xi, u_{i, j} \right\rangle u_{i, j} = \sum_{j \in J_i} \left\langle \xi, \Gamma_i^* u_{i, j} \right\rangle u_{i, j} = \sum_{j \in J_i} \left\langle \xi, w_{i, j} \right\rangle u_{i, j}.
	\]
	
	2. Similarly, for any \(g_i \in \mathcal{X}_i\), we can express:
	\[
	\Gamma_i^* g_i = \Gamma_i^*\left(\sum_{j \in J_i}\left\langle g_i, u_{i, j}\right\rangle u_{i, j}\right) = \sum_{j \in J_i}\left\langle g_i, u_{i, j}\right\rangle \Gamma_i^* u_{i, j} = \sum_{j \in J_i}\left\langle g_i, u_{i, j}\right\rangle w_{i, j}.
	\]
\end{proof}

\begin{theorem}\label{thrm3.10}
	Consider the set \(\left\{\Gamma_i\right\}\) where each \(\Gamma_i\) is an element of \(Hom_{\mathcal{A}}^*\left(\mathcal{X}, \mathcal{X}_i\right)\) for every index \(i \in I\). Additionally, let's refer to the sequence \(\left(w_{i, j}\right)\) as described in  \ref{eq3.6}.
	
	With this setup, the sequence \(\left\{\Gamma_i\right\}_{i \in I}\) will serve as a \(K\)-g-Riesz basis for the space \(\mathcal{X}\) in reference to \(\left\{\mathcal{X}_i\right\}_{i \in I}\). This holds true if and only if the sequence \(\left\{w_{i, j}: i \in I, j \in J_i\right\}\) establishes itself as a \(K\)-Riesz basis within \(\mathcal{X}\).
\end{theorem}

\begin{proof}
	Consider the assumption that the sequence \(\left\{\Gamma_i\right\}_{i \in I}\) is a \(K\)-g-Riesz basis for \(\mathcal{X}\) in relation to \(\left\{\mathcal{X}_i\right\}_{i \in I}\). Given this assumption, we have:
	$$
	\left\{\xi \in \mathcal{X}: \Gamma_i \xi=0 \text { for all } i \in I\right\} \subset Ker\left(K^*\right),
	$$
	and constants \(C, D>0\) such that for all sequences \(\left\{g_i\right\}_{i \in I} \in l^2\left(\left\{\mathcal{X}_i\right\}_{i \in I}\right)\):
	$$
	C \sum_{i \in I}\left\langle g_i, g_i\right\rangle \leq\left\langle\sum_{i \in I} \Gamma_i^* g_i, \sum_{i \in I} \Gamma_i^* g_i\right\rangle \leq D \sum_{i \in I}\left\langle g_i, g_i\right\rangle.
	$$
	Recall that for every \(i \in I\), \(\left\{u_{i, j}: j \in J_i\right\}\) forms a standard orthonormal basis for \(\mathcal{X}_i\). Thus, every element \(g_i \in \mathcal{X}_i\) can be expressed as:
	$$
	g_i=\sum_{j \in J_i} \alpha_{i, j} u_{i, j},
	$$
	with coefficients \(\left\{\alpha_{i, j}: j \in J_i\right\}\) belonging to \(l^2\left(J_i\right)\). This leads to the equivalent condition:
	$$
	C\sum_{i \in I} \sum_{j \in J_i} \alpha_{i, j} \alpha_{i, j}^* \leq\left\langle\sum_{i \in I} \sum_{j \in J_i} \alpha_{i, j} w_{i, j}, \sum_{i \in I} \sum_{j \in J_i} \alpha_{i, j} w_{i, j}\right\rangle \leq D \sum_{i \in I} \sum_{j \in J_i} \alpha_{i, j} \alpha_{i, j}^*.
	$$
	From the relation \(\Gamma_i \xi=\sum_{j \in J_i}\left\langle \xi, w_{i, j}\right\rangle u_{i, j}\) for all \(i \in I\), it becomes evident that:
	$$
	\left\{\xi \in \mathcal{X}: \Gamma_i \xi=0\right\} \text { for all } i \in I = \left\{\xi \in \mathcal{X}:\left\langle \xi, w_{i, j}\right\rangle=0\right\} \text { for all } i \in I \text { and } j \in J_i.
	$$
	From this, it follows that:
	$$
	\left\{\xi \in \mathcal{X}: \Gamma_i \xi=0 \text { for all } i \in I\right\} \subset Ker\left(K^*\right),
	$$
	if and only if:
	$$
	\left\{\xi \in \mathcal{X}:\left\langle \xi, w_{i, j}\right\rangle=0 \text { for all } i \in I \text { and } j \in J_i\right\} \subset Ker \left(K^*\right).
	$$
	In conclusion, the sequence \(\left\{\Gamma_i\right\}_{i \in I}\) acts as a \(K\)-g-Riesz basis for \(\mathcal{X}\) in relation to \(\left\{\mathcal{X}_i\right\}_{i \in I}\) if and only if the set \(\left\{w_{i, j}: i \in I, j \in J_i\right\}\) is a \(K\)-Riesz basis for \(\mathcal{X}\).

\end{proof}

\begin{theorem}\label{thrm3.10}
	Consider the set \(\left\{\Gamma_i\right\}\) where each \(\Gamma_i\) belongs to \(Hom_{\mathcal{A}}^*\left(\mathcal{X}, \mathcal{X}_i\right)\) for each \(i \in I\), and the sequence \(\left\{w_{i, j}\right\}\) as outlined in
	 \ref{eq3.6}. It is found that  \(\left\{\Gamma_i\right\}_{i \in I}\) serves as a g-orthonormal basis for \(\mathcal{X}\), relative to \(\left\{\mathcal{X}_i\right\}_{i \in I}\), if and only if the sequence \(\left\{w_{i, j}: i \in I, j \in J_i\right\}\) establishes itself as a conventional orthonormal basis for \(\mathcal{X}\).
\end{theorem}

\begin{proof}
	$\quad(\Rightarrow)$ Suppose the sequence $\left\{\Gamma_i\right\}_{i \in I}$ serves as a $g$-orthonormal basis for $\mathcal{X}$ with respect to $\left\{\mathcal{X}_i\right\}_{i \in I}$. Given references \ref{eq3.6} and \ref{eq2.3}, we deduce that
	$$
	\begin{aligned}
	\left\langle w_{i_1, j_1}, w_{i_2, j_2}\right\rangle & =\left\langle\Gamma_{i_1}^* u_{i_1, j_1}, \Gamma_{i_2}^* u_{i_2, j_2}\right\rangle \\
	& =\delta_{i_1, i_2}\left\langle u_{i_1, j_1}, u_{i_2, j_2}\right\rangle \\
	& =\delta_{i_1, i_2} \delta_{j_1, j_2},
	\end{aligned}
	$$
	for every choice of $i_1, i_2 \in I$, $j_1 \in J_{i_1}$, and $j_2 \in J_{i_2}$.
	Thus, the set $\left\{w_{i, j}: i \in I, j \in J_i\right\}$ qualifies as a standard orthonormal basis.
	Additionally, note that
	$$
	\begin{aligned}
	& \text { For each } \xi \in \mathcal{X}, \quad\langle \xi, \xi\rangle=\sum_{i \in I}\left\langle\Gamma_i \xi, \Gamma_i \xi\right\rangle \\
	& =\sum_{i \in I}\left\langle\sum_{j \in J_i}\left\langle \xi, w_{i, j}\right\rangle u_{i, j}, \sum_{j \in J_i}\left\langle \xi, w_{i, j}\right\rangle u_{i, j}\right\rangle \\
	& =\sum_{i \in I} \sum_{j \in J_i}\left\langle \xi, w_{i, j}\right\rangle\left\langle w_{i, j}, \xi\right\rangle \text {. } \\
	\end{aligned}
	$$
	Consequently, $\left\{w_{i, j}: i \in I, j \in J_i\right\}$ is recognized as a standard orthonormal basis for $\mathcal{X}$.
	
	$(\Leftarrow)$ To validate this direction, it suffices to verify \ref{eq2.3}. Indeed, from \ref{eq3.8}, for distinct $i, j \in I$ and arbitrary $g_i \in \mathcal{X}_i$ and $g_j \in \mathcal{X}_j$,
	$$
	\left\langle \Gamma_i^* g_i, \Gamma_j^* g_j\right\rangle=\left\langle\sum_{k \in J_i}\left\langle g_i, u_{i, k}\right\rangle w_{i, k}, \sum_{l \in J_j}\left\langle g_j, u_{j, l}\right\rangle w_{j, l}\right\rangle=0,
	$$
	and for each $g_i \in \mathcal{X}_i$,
	$$
	\left\langle \Gamma_i^* g_i, \Gamma_i^* g_i\right\rangle=\left\langle\sum_{k \in J_i}\left\langle g_i, u_{i, k}\right\rangle w_{i, k}, \sum_{l \in J_i}\left\langle g_i, u_{i, l}\right\rangle w_{i, l}\right\rangle=\left\langle g_i, g_i\right\rangle .
	$$
	Hence, $\left\{\Gamma_i\right\}_{i \in I}$ functions as a $g$-orthonormal basis for $\mathcal{X}$ when considered with respect to $\left\{\mathcal{X}_i\right\}_{i \in I}$.
\end{proof}

\begin{theorem}\label{thrm3.11}
	Given the sequence $\left\{\Gamma_i \in Hom_{\mathcal{A}}^*\left(\mathcal{X}, \mathcal{X}_i\right): i \in I\right\}$, it acts as a $K$-g-Riesz basis for $\mathcal{X}$ in association with $\left\{\mathcal{X}_i\right\}_{i \in I}$ if and only if there is a g-orthonormal basis  which we denote by	
	$\left\{\Xi_i \in Hom_{\mathcal{A}}^*\left(\mathcal{X}, \mathcal{X}_i\right): i \in I\right\}$ for $\mathcal{X}$, along with a bounded surjective operator $T \in Hom_{\mathcal{A}}^*(\mathcal{X})$ so that $\Gamma_i=\Xi_i T$ for every $i \in I$, and $Ran(K) \subset Ran\left(T^*\right)$.
\end{theorem}

\begin{proof}
	For each $i \in I$, let $\left\{u_{i, j}: j \in J_i\right\}$ be the standard orthonormal basis of $\mathcal{X}_i$. Assume that the set $\left\{\Gamma_i\right\}_{i \in I}$ forms a $K$-g-Riesz basis for $\mathcal{X}$ relative to $\left\{\mathcal{X}_i\right\}_{i \in I}$. By invoking Theorem \ref{thrm3.9}, we can identify a $K$-Riesz basis $\left\{w_{i, j}: i \in I, j \in J_i\right\}$ for $\mathcal{X}$ such that
	\[
	\Gamma_i \xi = \sum_{j \in J_i} \left\langle \xi, w_{i, j} \right\rangle u_{i, j}, \text{ for every } i \in I \text{ and } \xi \in \mathcal{X}.
	\]
	
	Now, consider the standard orthonormal set $\left\{y_{i, j}: i \in I, j \in J_i\right\}$ for $\mathcal{X}$. Owing to the fact that $\left\{w_{i, j}: i \in I, j \in J_i\right\}$ is a $K$-Riesz basis for $\mathcal{X}$, Theorem \ref{thrm3.3} guarantees the existence of a bounded below operator $\Phi \in Hom_{\mathcal{A}}^*(\mathcal{X})$ such that
	\[
	\Phi y_{i, j} = w_{i, j}, \text{ for all } i \in I \text{ and } j \in J_i,
	\]
	and furthermore, $Ran(K) \subset Ran(\Phi)$. Setting $T = \Phi^*$, we find from Proposition \ref{prop2.2} that $T$ is a bounded, surjective operator from $\mathcal{X}$ to $\mathcal{X}$ and $Ran(K) \subset Ran\left(T^*\right)$.
	
	For each $i \in I$, let's choose $\Xi_i$ in $Hom_{\mathcal{A}}^*\left(\mathcal{X}, \mathcal{X}_i\right)$ such that
	\[
	\Xi_i g = \sum_{j \in J_i} \left\langle g, y_{i, j} \right\rangle u_{i, j}, \text{ for every } g \in \mathcal{X}.
	\]
	According to Theorem \ref{thrm3.10},  $\left\{\Xi_i\right\}_{i \in I}$ serves as a $g$-orthonormal basis for $\mathcal{X}$. Additionally, for each $\xi$ in $\mathcal{X}$ and $i$ in $I$, we have:
	
	\[
	\Xi_i T \xi = \Gamma_i \xi.
	\]
	Thus, for all $i \in I$, $\Gamma_i = \Xi_i T$.
	
	On the converse side, if we recognize $\left\{\Xi_i\right\}_{i \in I}$ as a $g$-orthonormal basis for $\mathcal{X}$ and there exists a bounded surjective operator $T$ on $\mathcal{X}$ such that for all $i \in I$, $\Gamma_i = \Xi_i T $ and the range of $Ran(K)$ is included in  $Ran(T^*)$, then the subsequent relations hold:
	
	\begin{align*}
	\left\langle\Xi_i^* g_i, \Xi_j^* g_j\right\rangle &= \delta_{i, j}\left\langle g_i, g_j\right\rangle, \\
	\sum_{i \in I} \left\langle\Xi_i \xi, \Xi_i\xi\right\rangle &= \langle \xi, \xi \rangle.
	\end{align*}
	
	If, for all $i$ in $I$ and $\xi$ in $\mathcal{X}$, $\Gamma_i \xi = 0$, then $\xi$ belongs to the kernel of $T$. Given the conditions mentioned, and referencing Theorem \ref{thrm3.10}, we can identify a canonical orthonormal set $\left\{y_{i, j}: i \in I, j \in J_i\right\}$ for $\mathcal{X}$. Thus,
	\[
	\Gamma_i \xi = \sum_{j \in J_i}\left\langle \xi, T^* y_{i, j}\right\rangle u_{i, j}, \text{ for every } i \in I \text{ and } \xi \in \mathcal{X}.
	\]
	
	Given that $T^*$ is bounded below and $Ran(K) \subset Ran\left(T^*\right)$ as per Proposition \ref{prop2.2}, Theorem \ref{thrm3.3} implies that $\left\{T^* y_{i, j}: i \in I, j \in J_i\right\}$ is a $K$-Riesz basis for $\mathcal{X}$. Concluding from Theorem \ref{thrm3.9},  $\left\{\Gamma_i\right\}_{i \in I}$ indeed forms a $K$-g-Riesz basis for $\mathcal{X}$ relative to $\left\{\mathcal{X}_i\right\}_{i \in I}$.
	
\end{proof}

\begin{theorem}
	Given a set $\left\{\Gamma_i \in Hom_{\mathcal{A}}^*\left(\mathcal{X}, \mathcal{X}_i\right): i \in I\right\}$ constituting a $K$-$g$-Riesz basis for $\mathcal{X}$ aligned with $\left\{\mathcal{X}_i\right\}_{i \in I}$, and provided that each $\mathcal{X}_i$ possesses a Riesz basis represented by $\left\{z_{i, j}\right\}_{j \in P_i}$, for every $i \in I$, accompanied by bounded entities $A_i$ and $B_i$ satisfying $0 < \underset{i}{\text{inf}}\ A_i$ and $\underset{i}{\text{sup}}\ B_i < \infty$, with $P_i$ being a subset of $\mathbf{Z}$, it follows that the set $\left\{\Gamma_i^* z_{i, j}\right\}_{i \in I, j \in P_i}$ emerges as a $K$-Riesz basis for $\mathcal{X}$.
\end{theorem}

\begin{proof}
	 Since $\left\{\Gamma_i \in Hom_{\mathcal{A}}^*\left(\mathcal{X}, \mathcal{X}_i\right): i \in I\right\}$ is a $K$-g-Riesz basis for $\mathcal{X}$ with respect to $\left\{\mathcal{X}_i\right\}_{i \in I}$, then
$$\left\{\xi \in \mathcal{X}: \Gamma_i \xi=0\right., \text{ for all } \left.i \in I\right\} \subset Ker\left(K^*\right),$$
and there exist constants $C, D>0$ such that for all $\left\{g_i\right\}_{i \in I} \in l^2\left(\left\{\mathcal{\xi}_i\right\}_{i \in I}\right)$,
\begin{equation}\label{eq3.10}
C \sum_{i \in I}\left\langle g_i, g_i\right\rangle \leq\left\langle\sum_{i \in I} \Gamma_i^* g_i, \sum_{i \in I} \Gamma_i^* g_i\right\rangle \leq D \sum_{i \in I}\left\langle g_i, g_i\right\rangle .
\end{equation}
Moreover, since for all $i \in I,\left\{z_{i, j}\right\}_{j \in P_i}$ is a Riesz basis for $\mathcal{X}_i$,
$$
\left\{g_i \in \mathcal{X}_i:\left\langle g_i, z_{i, j}\right\rangle=0 \text {, for all } j \in P_i\right\}=\{0\},
$$
and for each $\left\{\alpha_{i, j}\right\}_{j \in P_i} \in l^2\left(P_i\right)$,

\begin{equation}\label{eq3.11}
A_i \sum_{j \in P_i} \alpha_{i, j} \alpha_{i, j}^* \leq\left\langle\sum_{j \in P_i} \alpha_{i, j} z_{i, j}, \sum_{j \in P_i} \alpha_{i, j} z_{i, j}\right\rangle \leq B_i \sum_{j \in P_i} \alpha_{i, j} \alpha_{i, j}^*, \text{ for all } i \in I. 
\end{equation}
We have
$$
\begin{aligned}
\left\{\xi \in \mathcal{X}:\left\langle\Gamma_i \xi, z_{i, j}\right\rangle\right. & \left.=0, \text { for all } i \in I \text { and } j \in P_i\right\} \\
& =\left\{\xi \in \mathcal{X}: \Gamma_i \xi=0, \text { for all } i \in I\right\} \\
& \subset Ker\left(K^*\right) .
\end{aligned}
$$
Thus $\left\{\xi \in \mathcal{X}:\left\langle \xi, \Gamma_i^* z_{i, j}\right\rangle=0\right.$, for all $i \in I$ and $\left.j \in P_i\right\} \subset Ker\left(K^*\right)$. If $\inf _i A_i=A$ and $\sup _i B_i=B$ then by \ref{eq3.10} and \ref{eq3.11}, for all $\left\{\gamma_{i, j}\right\}_{i \in I, j \in P_i} \in l^2(I)$, we have
$$
\begin{aligned}
C A  \sum_{i \in I} \sum_{j \in P_i} \gamma_{i, j} \gamma_{i, j}^* & \leq C \sum_{i \in I}\left\langle\sum_{j \in P_i} \gamma_{i, j} z_{i, j}, \sum_{j \in P_i} \gamma_{i, j} z_{i, j}\right\rangle \\
& \leq\left\langle\sum_{i \in I} \Gamma_i^*\left(\sum_{j \in P_i} \gamma_{i, j} z_{i, j}\right), \sum_{i \in I} \Gamma_i^*\left(\sum_{j \in P_i} \gamma_{i, j} z_{i, j}\right)\right\rangle \\
& \leq D \sum_{i \in I}\left\langle\sum_{j \in P_i} \gamma_{i, j} z_{i, j}, \sum_{j \in P_i} \gamma_{i, j} z_{i, j}\right\rangle \\
& \leq D B  \sum_{i \in I} \sum_{j \in P_i} \gamma_{i, j} \gamma_{i, j}^* .
\end{aligned}
$$
Thus $\left\{\Gamma_i^* z_{i, j}\right\}_{i \in I, j \in P_i}$ is a $K$-Riesz basis for $\mathcal{X}$.
\end{proof}

\begin{theorem}
	Let $\{\Gamma_i\}_{i \in I}$ be a sequence which acts as a $K$-$g$-frame for $\mathcal{X}$, associated with $\{\mathcal{X}_i\}_{i \in I}$. Suppose we have a limited subset $\sigma \subset I$. If the sequence $\{\Gamma_i\}_{i \in I \setminus \sigma}$ constitutes a $K$-$g$-Riesz basis for $\mathcal{X}$ in relation to $\{\mathcal{X}_i\}_{i \in I \setminus \sigma}$ and the series $\sum_{i \in I} \Gamma_i^* g_i$ shows convergence, it follows that  $\{g_i\}_{i \in I}$ belongs to $l^2\{\mathcal{X}_i\}_{i \in I}$.
\end{theorem}

\begin{proof}
	Start by observing that the series $\sum_{i \in I} \Gamma_i^* g_i$ converges whenever $g_i$ is an element of $\mathcal{X}_i$.
	
	With this premise, the convergence of the partial sum $\sum_{i \in I \setminus \sigma} \Gamma_i^* g_i$ naturally follows. Recognizing that $\{\Gamma_i\}_{i \in I \setminus \sigma}$ serves as a $K$-g-Riesz basis for $\mathcal{X}$ in association with $\{\mathcal{X}_i\}_{i \in I \setminus \sigma}$, and by invoking Theorem \ref{thrm3.11}, we ascertain the existence of a definite bounded operator $T: \mathcal{X} \to \mathcal{X}$ alongside a g-orthonormal ensemble $\{\Xi_i\}_{i \in I \setminus \sigma}$ for $\mathcal{X}$. This signifies that for each index $i$ outside the set $\sigma$, the relation $\Gamma_i=\Xi_i T$ holds true, with the scope of $K$ lying within $Ran(T^*)$.
	
	This allows the partial sum to be articulated as:
	$$
	\sum_{i \in I \setminus \sigma} \Gamma_i^* g_i = T^*\left( \sum_{i \in I \setminus \sigma} \Xi_i^* g_i \right).
	$$
	Given that the ensemble $\{\Xi_i\}_{i \in I \setminus \sigma}$ is g-orthonormally structured, the equation becomes:
	$$
	\sum_{i \in I \setminus \sigma}\langle g_i, g_i\rangle = \langle \sum_{i \in I \setminus \sigma} \Xi_i^* g_i, \sum_{i \in I \setminus \sigma} \Xi_i^* g_i \rangle < \infty.
	$$
	Concluding from the above, the collection $\{g_i\}_{i \in I \setminus \sigma}$ can be placed in $l^2\{\mathcal{X}_i\}_{i \in I \setminus \sigma}$, which implies that the entirety of $\{g_i\}_{i \in I}$ is nestled within $l^2\{\mathcal{X}_i\}_{i \in I}$.
\end{proof}

\end{document}